\documentclass[11pt]{article}
\usepackage{amssymb,amsmath,amsthm,enumerate,cite,geometry,float}
\newtheorem{theorem}{Theorem}
\newtheorem{proposition}[theorem]{Proposition}

\newtheorem{lemma}[theorem]{Lemma}
\newtheorem{corollary}[theorem]{Corollary}

\usepackage{lineno}
\usepackage{tikz}

\pgfdeclarelayer{nodelayer}
\pgfdeclarelayer{edgelayer}
\pgfsetlayers{edgelayer,nodelayer,main} 
\tikzstyle{white node}=[fill=white, draw=black, shape=circle, inner sep=0pt, minimum size=0.1cm]
\tikzstyle{black node}=[fill=black, draw=black, shape=circle, fill=black, draw=black, inner sep=0pt, minimum size=0.1cm]
\tikzstyle{none}=[inner sep=0pt]
\tikzstyle{grey dashed edge}=[-, dash pattern=on 2mm off 2mm, thick, dashed, draw={rgb,255: red,191; green,191; blue,191}]
\tikzstyle{black edge}=[-, fill=none, thick]

\usetikzlibrary{positioning}
\usetikzlibrary{graphs}
\usetikzlibrary{calc}
\usepackage{setspace}
\allowdisplaybreaks

\begin{document}
\onehalfspace
\title{Majority additive coloring and the maximum degree}
\author{Christoph Brause$^1$ \and Dieter Rautenbach$^2$ \and Laurin Schwartze$^2$}
\date{}
\maketitle
\vspace{-1cm}
\begin{center}
$^1$ TU Bergakademie Freiberg, 09596 Freiberg, Germany\\
\texttt{brause@math.tu-freiberg.de}\\[3mm]
$^2$ Ulm University, Ulm, Germany\\
\texttt{$\{$dieter.rautenbach,laurin.schwartze$\}$@uni-ulm.de}
\end{center}

\begin{abstract}
Kamyczura introduced the notion of a majority additive $k$-coloring of a graph $G$
as a function $c: V(G) \to \{1,2,\ldots,k\}$ such that
$$\left|\left\{u \in N_G(v):\sum_{w \in N_G(u)} c(w) = s \right\}\right|\leq \max\left\{1,\frac{d_G(v)}{2}\right\}$$
for every vertex $v$ of $G$ and every positive integer $s$. 
We show that every graph $G$ of maximum degree $\Delta$ admitting a majority additive coloring 
has a majority additive $\mathcal{O}\left(\Delta^2\right)$-coloring.
Under additional restrictions 
we improve this to sublinear in $\Delta$. 
We show that determining whether a majority additive $k$-coloring exists for a given graph 
is \textbf{NP}-complete for all $k\geq 2$.\\[3mm]
{\bf Keywords}: Majority additive coloring; majority coloring; additive coloring; lucky labeling; unfriendly partition
\end{abstract}

\section{Introduction}\label{section1}
We use standard graph theoretic notation. 
For a graph $G$, denote by $n(G)$, $\Delta(G)$, and $\delta(G)$ 
the order, the maximum degree, and the minimum degree of $G$, respectively. 
We write $[k]$ for $\{1,2,\ldots,k\}$ and $[k]_0$ for $[k] \cup \{0\}$, respectively.

Kamyczura \cite{ka1} introduced the notion of majority additive colorings as follows.
Let $G$ be a graph, let $k$ be a positive integer, let $c:V(G)\to [k]$, and let 
$$s_c:V(G)\to\mathbb{N}:u\mapsto \sum\limits_{v\in N_G(u)}c(v).$$
Now, the function $c$ is a {\it majority additive $k$-coloring} of $G$ if
there is no vertex $u$ of $G$ and no integer $s$ such that 
\begin{itemize}
\item the degree of $u$ is at least $2$, and 
\item $s_c(v)=s$ for more than half the neighbors $v$ of $u$.
\end{itemize}
Majority additive colorings are a variation of some well studied notions that received considerable attention.
Majority colorings of the vertices or edges of finite or infinite graphs \cite{lo,bo,ha,ka,pe} or digraphs \cite{kr,kn,an}
are well studied objects with challenging open problems.
Similarly, additive colorings based on labelings of the vertices or edges of graphs 
were investigated \cite{kar,cz,ba} also motivated by challenging problems \cite{ke,pr,ad,kal}.

As pointed out by Kamyczura, 
a graph $G$ does not have a majority additive $k$-coloring for any choice of $k$ if and only if
\begin{eqnarray}\label{e1}
\exists u\in V(G):\exists R\subseteq N_G(u):2|R|>d_G(u)>1\mbox{ and }\forall v,w\in R:N_G(v)=N_G(w).
\end{eqnarray}
In fact, a vertex $u$ as in (\ref{e1}) is clearly an obstruction 
to the existence of a majority additive $k$-coloring $c$ of $G$, 
because $s_c(v)=s_c(w)$ for all vertices $v$ and $w$ in the set $R$.
On the other hand, if (\ref{e1}) does not hold and $V(G)=\{ u_1,\ldots,u_n\}$,
then setting $c(u_i)=2^{i-1}$ for $i$ in $[n]$
yields a majority additive $2^{n-1}$-coloring $c$ of $G$.
We call a graph $G$ {\it good} if (\ref{e1}) does not hold.
The {\it majority additive chromatic number $\chi_{mac}(G)$} of a good graph $G$
is the smallest $k$ such that $G$ has a majority additive $k$-coloring.

Kamyczura \cite{ka2} conjectured that 
$\chi_{mac}(G)\leq \Delta(G)+1$ 
for every good graph $G$.
He even conjectured \cite{ka2} that 
$\chi_{mac}(G)\leq 3$ 
for every good graph $G$ distinct from $K_4$. 
In Section \ref{sec2} we disprove these conjectures 
but show that
$\chi_{mac}(G) \leq 2\Delta(G)(\Delta(G)-1)+1$ for every good graph $G$  
and that 
$\chi_{mac}(G) \leq 4e^3\cdot\Delta(G)^{\frac{4}{\left\lfloor\delta(G)/2\right\rfloor}}$ 
for every good graph $G$ satisfying an additional local condition. 
This leads us to believe that $\chi_{mac}(G) \in \mathcal{O}(\Delta(G))$ for every good graph $G$. 
Section \ref{sec3} is dedicated to the complexity of $\chi_{mac}(G)$. 
We show that $k$-\textsc{Mac}, 
the problem to decide whether $\chi_{mac}(G) \leq k$ for a given graph $G$, 
is \textbf{NP}-complete for all $k\geq 2$.

\section{Bounds on $\chi_{mac}$}\label{sec2}
As mentioned in the introduction, we begin by disproving Kamyczura's conjectures.

\begin{proposition}
There are infinitely many good graphs $G$ with $\chi_{mac}(G) \geq 2\Delta(G)+1$.
\end{proposition}
\begin{proof}
Let $n$ be a positive integer at least $7$ that is $1$ or $3$ modulo $6$.
Let the hypergraph $H$ with vertex set $\{ u_1,\ldots,u_n\}$ 
be a {\it Steiner triple system $S(2,3,n)$},
that is, $H$ is a $3$-uniform hypergraph 
such that, for every two distinct $x$ and $y$ in $[n]$,
there is exactly one hyperedge $e$ of $H$ that contains $u_x$ and $u_y$.
It is well known that the parity conditions on $n$ imply the existence of $S(2,3,n)$.
Now, let the graph $G$ arise from $H$ by replacing 
every hyperedge $e=\{ u_x,u_y,u_z\}$
by four new vertices 
$v_x^e$, $v_y^e$, $v_z^e$, and $w^e$
and six new edges 
$u_xv_x^e$,
$u_yv_y^e$,
$u_zv_z^e$,
$v_x^ew^e$,
$v_y^ew^e$, and
$v_z^ew^e$.
If $c:V(G)\to [k]$ is such that $c(u_x)=c(u_y)$ for distinct $x$ and $y$ in $[n]$
and $e=\{ u_x,u_y,u_z\}$ is a hyperedge of $H$,
then $s_c(v_x^e)=s_c(v_y^e)$,
that is, two of the three neighbors of $w^e$
have the same $s_c$-value,
which implies that $c$ is not a majority additive $k$-coloring.
It follows that 
$\chi_{mac}(G)\geq n=2\Delta(G)+1$.
\end{proof}
Next, we show that $\chi_{mac}$ grows at most quadratically in the maximum degree. 
This is the first sub-exponential upper bound 
as Kamyczura \cite{ka1} proved $\chi_{mac}(G) \leq 2^{n(G)}$ for every good graph $G$.
\begin{theorem}\label{theorem1}
If $G$ is a good graph of maximum degree $\Delta$, 
then $\chi_{mac}(G)\leq 2\Delta(\Delta-1)+1$.
\end{theorem}
\begin{proof}
Let $V(G)=\{ u_1,\ldots,u_n\}$ and initialize $c(u_i)$ as $2^{i-1}$ for $i$ in $[n]$.
As noted above, $c$ is a majority additive coloring of $G$.
The proof relies on a greedy (re-)coloring procedure:
\begin{quote}
{\it Consider the vertices $u$ of $G$ in any linear order and
greedily change the value $c(u)$ to the smallest positive integer 
such that $c$ remains a majority additive coloring of $G$.}
\end{quote}
In order to complete the proof, 
we argue that each $c(u)$ is at most $2\Delta(\Delta-1)+1$ after the change.

Let $c$ be some intermediate majority additive coloring of $G$.
Suppose that $u$ is a vertex whose $c$-value our greedy procedure currently changes
and that $w$ is a vertex that is affected by this change,
that is, changing $c(u)$ to $c(u)+\delta$ may cause more than half the neighbors of $w$ to have the same $s_c$-value.
This necessarily requires that the two sets
$A=N_G(u)\cap N_G(w)$ 
and $B=N_G(w)\setminus A$ are both non-empty.
Let 
${\cal A}$ and ${\cal B}$ be the partitions of $A$ and $B$, respectively,
into the sets of the equal $s_c$-values,
that is,
${\cal A}$ is the partition of $A$ 
and
${\cal B}$ is the partition of $B$ 
that are induced by the partition of $V(G)$ 
given by the non-empty sets of the form $s_c^{-1}(k)$ with $k\in\mathbb{N}$.
Let 
$|{\cal A}|=p$,
$|{\cal B}|=q$,
\begin{eqnarray*}
{\cal A}&=&\{ A_1,\ldots,A_p\}\mbox{ with $|A_1|\geq \ldots \geq |A_p|$, and }\\
{\cal B}&=&\{ B_1,\ldots,B_q\}\mbox{ with $|B_1|\geq \ldots \geq |B_q|$.}
\end{eqnarray*}
Let $c':V(G)\to\mathbb{N}$ arise from $c$ 
by changing $c(u)$ to $c(u)+\delta$ for some $\delta\in\mathbb{Z}$.
Note that $s_{c'}(v)=s_{c}(v)+\delta$ for every $v$ in $A$
and $s_{c'}(v)=s_{c}(v)$ for every $v$ in $B$.

Now, more than half the neighbors of $w$ have the same $s_{c'}$-value, say $s$,
if and only if there is some $i$ in $[p]$ and some $j$ in $[q]$ such that 
$s_{c'}(v)=s$ for every $v$ in $A_i\cup B_j$ and
$|A_i|+|B_j|>d_G(w)/2$.
In particular, we have $\delta=s_{c}(v')-s_{c}(v)$ 
for every $v$ in $A_i$ and $v'$ in $B_j$,
that is, the specific value of $\delta$ that leads to this violation is uniquely
determined by the indices $i$ and $j$.
Let 
$$P=\left\{ (i,j)\in [p]\times [q]:|A_i|+|B_j|>\frac{d_G(w)}{2}\right\}.$$
Since 
$$d_G(w)=|A|+|B|=|A_1|+\cdots+|A_p|+|B_1|+\cdots+|B_q|,$$
it follows easily that
\begin{itemize}
\item either $P\subseteq \{ (1,j):j\in [q]\}$,
in which case we say that $w$ is of {\it type $1$}
\item or $P\subseteq \{ (i,1):i\in [p]\}$ and $P\not=\{ (1,1)\}$,
in which case we say that $w$ is of {\it type $2$}.
\end{itemize}
Our next goal is to bound $|P|/|A|$.

Let $d=d_G(w)$.

First, suppose that $w$ is of type $1$.
Let $\ell=|A_1|$.
This implies $|B|\leq d-\ell$.
Since $c$ is a majority additive coloring of $G$, 
we have $1\leq \ell\leq \left\lfloor\frac{d}{2}\right\rfloor$.
If $(1,j)\in P$, then $|A_1|+|B_j|=\ell+|B_j|>d/2$,
which implies $|B_j|\geq \left\lfloor\frac{d}{2}\right\rfloor-\ell+1$,
and, hence, $|P|\leq \frac{d-\ell}{\left\lfloor\frac{d}{2}\right\rfloor-\ell+1}$.
Therefore, we obtain that
\begin{eqnarray}\label{e2}
\frac{|P|}{|A|}&\leq &\max\left\{\frac{d-\ell}{\ell\left(\left\lfloor\frac{d}{2}\right\rfloor-\ell+1\right)}:
1\leq \ell\leq \left\lfloor\frac{d}{2}\right\rfloor\right\}
=\frac{d-1}{\left\lfloor\frac{d}{2}\right\rfloor}\leq 2.
\end{eqnarray}
Next, suppose that $w$ is of type $2$.
Let $d-\ell=|B_1|$.
This implies $|A|\leq \ell$.
Since $c$ is a majority additive coloring of $G$, 
we have $1\leq d-\ell\leq \left\lfloor\frac{d}{2}\right\rfloor$,
and, hence, $\left\lceil\frac{d}{2}\right\rceil\leq \ell\leq d-1$.
If $(i,1)\in P$, then $|A_i|+|B_1|=|A_i|+d-\ell>d/2$,
which implies $|A_i|\geq \ell-\left\lceil\frac{d}{2}\right\rceil+1$,
and, hence, 
$|P|\leq \frac{|A|}{\ell-\left\lceil\frac{d}{2}\right\rceil+1}\leq \frac{\ell}{\ell-\left\lceil\frac{d}{2}\right\rceil+1}$.
Therefore, we obtain that
\begin{eqnarray}\label{e3}
\frac{|P|}{|A|}&\leq &\max\left\{\frac{1}{\ell-\left\lceil\frac{d}{2}\right\rceil+1}:
\left\lceil\frac{d}{2}\right\rceil\leq \ell\leq d-1\right\}
=1.
\end{eqnarray}
Recall that, when considering the vertex $u$, 
our greedy approach explained at the beginning of the proof 
changes $c(u)$ to the smallest positive integer $c(u)+\delta$
such that $c$ remains a majority additive coloring of $G$.
Therefore, we need to bound the number of values $\delta$ in $\mathbb{Z}$
such that changing $c(u)$ to $c(u)+\delta$ 
leads to a violation within the neighborhood of some vertex $w$.
If this is the case, then, clearly, the two sets 
$A=N_G(u)\cap N_G(w)$ and $B=N_G(w)\setminus A$ 
are both non-empty and --- as explained above --- the vertex $w$ 
is either of type 1 or type 2.
Let $P$ be as above.
Note that there are at most $|P|$ values of $\delta$
such that changing $c(u)$ to $c(u)+\delta$
leads to a violation at $w$,
that is, to the situation that more than half the neighbors of $w$ 
have the same $s_c$-value after the change.
In order to properly count the total number of forbidden values for $\delta$,
we charge these at most $|P|$ forbidden values 
caused by violations at $w$ 
to the $|A|$ pairs $(w,e)$ where $e$ is an edge between $w$ and $A$.
By (\ref{e2}) and (\ref{e3}), 
the charge $\frac{|P|}{|A|}$ on each pair $(w,e)$ is at most $2$.
Since $G$ has maximum degree $\Delta$,
there are at most $\Delta(\Delta-1)$ pairs $(w,e)$
such that $w$ is a vertex of $G$
and $e$ is an edge of $G$ between $w$ and some neighbor of $u$.
Since each such pair receives at most a charge of $2$,
the total charge over all possible pairs $(w,e)$ is at most $2\Delta(\Delta-1)$.
This implies that when our greedy (re-)coloring procedure considers $u$,
there are at most $2\Delta(\Delta-1)$ values of the change $\delta$ of $c(u)$
that lead to some violation.
Hence, at the end of the (re-)coloring,
we have $c(u)\leq 2\Delta(\Delta-1)+1$ for every vertex $u$ of $G$ as claimed.
\end{proof}

A graph $G$ is said to satisfy the \textit{private neighbor condition} if 
\begin{eqnarray}\label{e4}
\forall u \in V(G): d_G(u)\geq 2\Rightarrow \forall v \in N_G(u):\exists w \in N_G(v): N_G(w) \cap N_G[u] = \{ v\},
\end{eqnarray} 
in which case we say that $w$ is a \textit{private neighbor} of $v$ in $N_G(u)$. 
Restricting to graphs that satisfy \eqref{e4}, 
we can improve the upper bound to sublinear in $\Delta(G)$ for graphs $G$ with large enough minimum degree. 

\begin{theorem}\label{theorem2}
Let $G$ be a graph of maximum degree $\Delta$ that satisfies the private neighbor condition \eqref{e4} and 
let $\delta = \min\left\{d_G(u):u \in V(G)\mbox{ and }d_G(u)\geq 2\right\}$. 
Then $G$ is good and $\chi_{mac}(G) \leq 4e^3 \cdot \Delta^{\frac{4}{\left\lfloor\delta/2\right\rfloor}}$.
\end{theorem}

For the proof, we need two well-known results.
\begin{lemma}[Extended Markov's inequality]\label{markov}
Let $X$ be a non-negative random variable, 
let $a>0$, and 
let $\varphi$ be a non-decreasing function such that $\varphi(a)>0$. 
Then $\mathbb{P}\left[X \geq a\right] \leq \frac{\mathbb{E}[\varphi\left(X\right)]}{\varphi(a)}.$
\end{lemma}

\begin{lemma}[Symmetric Lovász-Local-Lemma]\label{LLL}
Let $X_1,X_2,\ldots,X_m$ be random events each occurring with probability at most $p$. 
If the maximum degree of the dependency graph is at most $d$ and $ep(d+1) \leq 1$,
then, with non-zero probability, none of the $X_i$ occurs.
\end{lemma}

\begin{proof}[Proof of Theorem \ref{theorem2}]
Clearly, the private neighbor condition \eqref{e4} implies that $G$ is good.

Let $k = \left\lceil 4e^3\cdot \Delta^{\frac{4}{\left\lfloor\delta/2\right\rfloor}}\right\rceil$. 
For every vertex $u$ of $G$, choose $c(u) \in [k]$ uniformly at random and define 
$$\hat s_c(u) = \sum_{v\in N_G(u)} c(v) \mod k.$$
Since 
$\left|\{v\in N_G(u):\hat s_c(v) \equiv s \mod k\}\right| \geq \left|\{v\in N_G(u):s_c(v) = s\}\right|$
for every vertex $u$ of $G$ and every integer $s$,
it suffices to show that, with positive probability, 
there is no vertex $u$ of degree at least $2$ 
such that more than half the neighbors of $u$ have the same $\hat s_c$-value.

For every integer $s$ in $[k-1]_0$ and every vertex $u$ of $V(G)$ of degree at least $2$,
let the random variable $X_s^u$ be the number of neighbors $v$ of $u$ with $\hat s_c(v) \equiv s \mod k$. 
Let $u$ be a vertex of degree at least $2$ and let $v$ be a neighbor of $u$.
Let $w$ be a private neighbor of $v$ in $N_G(u)$. 
For any integer $s$, the private neighbor condition \eqref{e4} implies that 
$$\hat s_c(v) \equiv s \mod k \iff c(w) \equiv s - \sum_{x \in N_G(v)\setminus \{w\}} c(x) \mod k.$$
Furthermore,
for every two distinct neighbors $v$ and $v'$ of $u$,
the events $\left[\hat s_c(v) \equiv s \mod k\right]$ and $\left[\hat s_c\left(v'\right) \equiv s \mod k\right]$ 
are pairwise independent
and occur with probability $\frac{1}{k}$,
which implies 
$X_s^u \sim \text{Bin}\left(d_G(u),\frac{1}{k}\right)$.
Lemma \ref{markov} with $\varphi$ chosen as the $r$-th factorial moment, implies 
\begin{eqnarray}\label{e6}
\mathbb{P}\left[X_s^u \geq \left\lfloor \frac{d_G(u)}{2}\right\rfloor + 1\right] 
\leq 
\frac{\mathbb{E}\left[\left(X_s^u\right)_r\right]}{\left(\left\lfloor \frac{d_G(u)}{2}\right\rfloor + 1\right)_r}.
\end{eqnarray}
It is easy to see that, for $Y \sim \text{Bin}(n,p)$, it holds that $\mathbb{E}[(Y)_r] = (n)_r\cdot p^r.$ 
Since $(r)_r=r!$, choosing $r = \left\lfloor \frac{d_G(u)}{2}\right\rfloor + 1$ in \eqref{e6} yields 
\begin{eqnarray*}
\mathbb{P}\left[X_s^u\geq \left\lfloor \frac{d_G(u)}{2}\right\rfloor + 1\right] 
\leq \frac{(d_G(u))_r}{r!\cdot k^r} 
= \frac{\binom{d_G(u)}{r}}{k^r} 
\leq \left(\frac{e\cdot d_G(u)}{k\cdot\left(\left\lfloor \frac{d_G(u)}{2}\right\rfloor + 1\right)}\right)^{\left\lfloor \frac{d_G(u)}{2}\right\rfloor + 1} 
\leq \left(\frac{2e}{k}\right)^{\left\lfloor \frac{d_G(u)}{2}\right\rfloor + 1}.
\end{eqnarray*}
Let $X_u$ be the event that there is some $s$ in $[k-1]_0$ 
such that more than half the neighbors of $u$ have $\hat s_c$-value $s$. 
Applying the union bound yields 
\begin{eqnarray*}
\mathbb{P}[X_u] 
&=& \sum_{s \in [k-1]_0} \mathbb{P}\left[X_s^u\geq \left\lfloor \frac{d_G(u)}{2}\right\rfloor + 1\right]\\ 
&\leq & k\cdot \left(\frac{2e}{k}\right)^{\left\lfloor \frac{d_G(u)}{2}\right\rfloor + 1} \\
&\leq & 2e\cdot \left(\frac{2e}{\left\lceil4e^3\cdot \Delta^{\frac{4}{\left\lfloor\delta/2\right\rfloor}}\right\rceil}\right)^{\left\lfloor\frac{\delta}{2}\right\rfloor}\\
& \leq & 2e\cdot\left(\frac{1}{\left(2e^2\right)^{\left\lfloor\delta/2\right\rfloor}\cdot \Delta^4}\right).
\end{eqnarray*}
An event $X_u$ is independent of the events $X_v$ for all $v$ of distance at least $5$ to $u$.
Therefore, the maximum degree of the dependency graph is at most 
$\Delta+\Delta(\Delta-1)+\Delta(\Delta-1)^2+\Delta(\Delta-1)^3
=\Delta^4-2\Delta^3+2\Delta^2$. 
By Lemma \ref{LLL}, a majority additive $k$-coloring exists if \[2e^2\cdot \frac{1}{\left(2e^2\right)^{\left\lfloor\delta/2\right\rfloor}\Delta^4}\cdot(\Delta^4-2\Delta^3+2\Delta^2+1) \leq 1\] which is satisfied as $\delta \geq 2$.
\end{proof}

\begin{corollary}
Let $G$ be a graph of maximum degree $\Delta$ and minimum degree $\delta \in \Omega(\log \Delta)$ that satisfies the private neighbor condition \eqref{e4}. Then $G$ is good and $\chi_{mac}(G)$ is bounded by a constant.
\end{corollary}

\section{Complexity of $\chi_{mac}$}\label{sec3}

In this section we discuss the complexity of the decision problem $k$-\textsc{Mac}.
Clearly, a graph $G$ has a majority additive $1$-coloring 
if and only if there is no vertex $u$ of $G$ such that more than half the neighbors of $u$ have the same degree.
Hence, $1$-\textsc{Mac} can be solved efficiently.

\begin{theorem}
$k$-\textsc{Mac} is \textbf{NP}-complete for every $k\geq 2$.
\end{theorem}
\begin{proof} 
Clearly, $k$-\textsc{Mac} is in \textbf{NP} for every $k$.

For $k=2$, we reduce from the \textbf{NP}-complete problem \textsc{Not-All-Equal $3$Sat}, cf. [LO3] in \cite{gajo}.
Let $F$ be an instance of \textsc{Not-All-Equal $3$Sat} consisting of $m$ clauses in $n$ boolean variables. 
For a literal $\ell$, let $d_\ell$ denote the number of clauses of $F$ that contain $\ell$.
Starting with the empty graph, we construct a graph $G$ as follows. 
See Figure \ref{fig4} for an illustration.
\begin{itemize}
\item For every variable $x$, 
\begin{itemize}
\item add the path $v_xa_xb_xa_{\bar{x}}v_{\bar{x}}$ of order five and 
\item add the two vertices $\alpha_i^\ell$ and $\beta_i^\ell$ as well as 
the two edges $\alpha_i^\ell\beta_i^\ell$ and $\beta_i^\ell v_\ell$ 
for every $\ell\in \{ x,\bar{x}\}$ and every $i\in [d_\ell+3]$.
\end{itemize}
\item For every clause $C$, 
add the path $v_C\beta_C\alpha_C$ of order three.
\item For every literal $\ell$ and every clause $C$ such that $C$ contains $\ell$, 
add the vertex $v_\ell^C$ as well as the two edges $v_\ell v_\ell^C$ and $v_C v_\ell^C$.
\end{itemize}

\begin{figure}[h]
	 	\centering
	 	\begin{tikzpicture}[scale=0.8]
        	\begin{pgfonlayer}{nodelayer}
        		\node [style=black node] (0) at (-6, 0) {};
        		\node [style=black node] (1) at (-3, 0) {};
        		\node [style=black node] (2) at (0, 0) {};
        		\node [style=black node] (3) at (3, 0) {};
        		\node [style=black node] (4) at (6, 0) {};
        		\node [style=black node] (5) at (-3, 0) {};
        		\node [style=none] (15) at (-6, 0.5) {$v_w$};
        		\node [style=none] (16) at (-3, 0.5) {$a_w$};
        		\node [style=none] (17) at (0, 0.5) {$b_w$};
        		\node [style=none] (18) at (3, 0.5) {$c_w$};
        		\node [style=none] (19) at (6, 0.5) {$v_{\bar w}$};
        		\node [style=black node] (20) at (-9, 2.5) {};
        		\node [style=black node] (21) at (-9, 2) {};
        		\node [style=black node] (22) at (-9, 1.5) {};
        		\node [style=black node] (23) at (8, 2.5) {};
        		\node [style=black node] (24) at (8, 2) {};
        		\node [style=black node] (25) at (8, 1.5) {};
        		\node [style=black node] (28) at (9.25, 2.5) {};
        		\node [style=black node] (29) at (9.25, 2) {};
        		\node [style=black node] (30) at (9.25, 1.5) {};
        		\node [style=black node] (31) at (-10.25, 2.5) {};
        		\node [style=black node] (32) at (-10.25, 2) {};
        		\node [style=black node] (33) at (-10.25, 1.5) {};
        		\node [style=black node] (34) at (-6, -3) {};
        		\node [style=black node] (35) at (-3, -3) {};
        		\node [style=black node] (36) at (0, -3) {};
        		\node [style=black node] (37) at (3, -3) {};
        		\node [style=black node] (38) at (6, -3) {};
        		\node [style=black node] (39) at (-3, -3) {};
        		\node [style=none] (40) at (-6, -2.5) {$v_x$};
        		\node [style=none] (41) at (-3, -2.5) {$a_x$};
        		\node [style=none] (42) at (0, -2.5) {$b_x$};
        		\node [style=none] (43) at (3, -2.5) {$c_x$};
        		\node [style=none] (44) at (6, -2.5) {$v_{\bar x}$};
        		\node [style=black node] (45) at (-9, -0.5) {};
        		\node [style=black node] (46) at (-9, -1) {};
        		\node [style=black node] (47) at (-9, -1.5) {};
        		\node [style=black node] (48) at (8, -0.5) {};
        		\node [style=black node] (49) at (8, -1) {};
        		\node [style=black node] (50) at (8, -1.5) {};
        		\node [style=black node] (53) at (9.25, -0.5) {};
        		\node [style=black node] (54) at (9.25, -1) {};
        		\node [style=black node] (55) at (9.25, -1.5) {};
        		\node [style=black node] (56) at (-10.25, -0.5) {};
        		\node [style=black node] (57) at (-10.25, -1) {};
        		\node [style=black node] (58) at (-10.25, -1.5) {};
        		\node [style=black node] (59) at (-6, -6) {};
        		\node [style=black node] (60) at (-3, -6) {};
        		\node [style=black node] (61) at (0, -6) {};
        		\node [style=black node] (62) at (3, -6) {};
        		\node [style=black node] (63) at (6, -6) {};
        		\node [style=black node] (64) at (-3, -6) {};
        		\node [style=none] (65) at (-6, -5.5) {$v_y$};
        		\node [style=none] (66) at (-3, -6.5) {$a_y$};
        		\node [style=none] (67) at (0, -5.5) {$b_y$};
        		\node [style=none] (68) at (3, -5.5) {$c_y$};
        		\node [style=none] (69) at (6, -5.5) {$v_{\bar y}$};
        		\node [style=black node] (70) at (-9, -3.5) {};
        		\node [style=black node] (71) at (-9, -4) {};
        		\node [style=black node] (72) at (-9, -4.5) {};
        		\node [style=black node] (73) at (8, -3.5) {};
        		\node [style=black node] (74) at (8, -4) {};
        		\node [style=black node] (75) at (8, -4.5) {};
        		\node [style=black node] (78) at (9.25, -3.5) {};
        		\node [style=black node] (79) at (9.25, -4) {};
        		\node [style=black node] (80) at (9.25, -4.5) {};
        		\node [style=black node] (81) at (-10.25, -3.5) {};
        		\node [style=black node] (82) at (-10.25, -4) {};
        		\node [style=black node] (83) at (-10.25, -4.5) {};
        		\node [style=black node] (84) at (-6, -9) {};
        		\node [style=black node] (85) at (-3, -9) {};
        		\node [style=black node] (86) at (0, -9) {};
        		\node [style=black node] (87) at (3, -9) {};
        		\node [style=black node] (88) at (6, -9) {};
        		\node [style=black node] (89) at (-3, -9) {};
        		\node [style=none] (90) at (-6, -8.5) {$v_z$};
        		\node [style=none] (91) at (-3, -8.5) {$a_z$};
        		\node [style=none] (92) at (-0.25, -8.5) {$b_z$};
        		\node [style=none] (93) at (3, -8.5) {$c_z$};
        		\node [style=none] (94) at (5.75, -8.5) {$v_{\bar z}$};
        		\node [style=black node] (95) at (-9, -6.5) {};
        		\node [style=black node] (96) at (-9, -7) {};
        		\node [style=black node] (97) at (-9, -7.5) {};
        		\node [style=black node] (98) at (8, -6.5) {};
        		\node [style=black node] (99) at (8, -7) {};
        		\node [style=black node] (100) at (8, -7.5) {};
        		\node [style=black node] (103) at (9.25, -6.5) {};
        		\node [style=black node] (104) at (9.25, -7) {};
        		\node [style=black node] (105) at (9.25, -7.5) {};
        		\node [style=black node] (106) at (-10.25, -6.5) {};
        		\node [style=black node] (107) at (-10.25, -7) {};
        		\node [style=black node] (108) at (-10.25, -7.5) {};
        		\node [style=black node] (109) at (-4, -12) {};
        		\node [style=black node] (110) at (0, -12) {};
        		\node [style=black node] (111) at (4, -12) {};
        		\node [style=black node] (112) at (-4, -13) {};
        		\node [style=black node] (113) at (-4, -14) {};
        		\node [style=black node] (114) at (0, -13) {};
        		\node [style=black node] (115) at (0, -14) {};
        		\node [style=black node] (116) at (4, -13) {};
        		\node [style=black node] (117) at (4, -14) {};
        		\node [style=black node] (118) at (0.5, -7.5) {};
        		\node [style=black node] (119) at (-1.25, -7.25) {};
        		\node [style=black node] (120) at (-5.25, -7) {};
        		\node [style=black node] (121) at (-5.5, -7.75) {};
        		\node [style=black node] (122) at (-5.25, -10.25) {};
        		\node [style=black node] (123) at (-2.5, -10.75) {};
        		\node [style=black node] (124) at (2.5, -4.75) {};
        		\node [style=black node] (125) at (3, -7.25) {};
        		\node [style=black node] (126) at (5.5, -7.5) {};
        		\node [style=none] (128) at (0.6, -7) {$v_w^{C_3}$};
        		\node [style=none] (130) at (-1.75, -7.5) {$v_x^{C_3}$};
        		\node [style=none] (132) at (-4.5, -7) {$v_x^{C_1}$};
        		\node [style=none] (134) at (-6.25, -7.5) {$v_y^{C_1}$};
        		\node [style=none] (136) at (-2.25, -10.25) {$v_z^{C_2}$};
        		\node [style=none] (138) at (-6, -10.25) {$v_z^{C_1}$};
        		\node [style=none] (139) at (-3, -12) {};
        		\node [style=none] (140) at (-3.25, -12) {$v_{C_1}$};
        		\node [style=none] (141) at (-3, -13) {};
        		\node [style=none] (142) at (-3.25, -13) {$\beta_{C_1}$};
        		\node [style=none] (143) at (-3.25, -14) {$\alpha_{C_1}$};
        		\node [style=none] (144) at (0.75, -12) {$v_{C_2}$};
        		\node [style=none] (145) at (0.75, -13) {$\beta_{C_2}$};
        		\node [style=none] (146) at (0.75, -14) {$\alpha_{C_2}$};
        		\node [style=none] (147) at (4.75, -12) {$v_{C_3}$};
        		\node [style=none] (148) at (4.75, -13) {$\beta_{C_3}$};
        		\node [style=none] (149) at (4.75, -14) {$\alpha_{C_3}$};
        		\node [style=none] (150) at (2, -4.5) {$v_{\bar w}^{C_2}$};
        		\node [style=none] (151) at (3.5, -7.5) {$v_{\bar x}^{C_2}$};
        		\node [style=none] (152) at (6.25, -7.25) {$v_{\bar y}^{C_3}$};
        		\node [style=black node] (153) at (-10.25, 1) {};
        		\node [style=black node] (154) at (-9, 1) {};
        		\node [style=black node] (155) at (-10.25, -2) {};
        		\node [style=black node] (156) at (-9, -2) {};
        		\node [style=black node] (157) at (-10.25, -2.5) {};
        		\node [style=black node] (158) at (-9, -2.5) {};
        		\node [style=black node] (159) at (-10.25, -5) {};
        		\node [style=black node] (160) at (-9, -5) {};
        		\node [style=black node] (161) at (-10.25, -8) {};
        		\node [style=black node] (162) at (-9, -8) {};
        		\node [style=black node] (163) at (-10.25, -8.5) {};
        		\node [style=black node] (164) at (-9, -8.5) {};
        		\node [style=black node] (165) at (8, 1) {};
        		\node [style=black node] (166) at (9.25, 1) {};
        		\node [style=black node] (167) at (8, -2) {};
        		\node [style=black node] (168) at (9.25, -2) {};
        		\node [style=black node] (169) at (8, -5) {};
        		\node [style=black node] (170) at (9.25, -5) {};
        	\end{pgfonlayer}
        	\begin{pgfonlayer}{edgelayer}
        		\draw [style=black edge] (0) to (5);
        		\draw [style=black edge] (5) to (2);
        		\draw [style=black edge] (2) to (3);
        		\draw [style=black edge] (3) to (4);
        		\draw [style=black edge] (20) to (0);
        		\draw [style=black edge] (21) to (0);
        		\draw [style=black edge] (22) to (0);
        		\draw [style=black edge] (23) to (4);
        		\draw [style=black edge] (24) to (4);
        		\draw [style=black edge] (25) to (4);
        		\draw [style=black edge] (31) to (20);
        		\draw [style=black edge] (32) to (21);
        		\draw [style=black edge] (33) to (22);
        		\draw [style=black edge] (23) to (28);
        		\draw [style=black edge] (24) to (29);
        		\draw [style=black edge] (25) to (30);
        		\draw [style=black edge] (34) to (39);
        		\draw [style=black edge] (39) to (36);
        		\draw [style=black edge] (36) to (37);
        		\draw [style=black edge] (37) to (38);
        		\draw [style=black edge] (45) to (34);
        		\draw [style=black edge] (46) to (34);
        		\draw [style=black edge] (47) to (34);
        		\draw [style=black edge] (48) to (38);
        		\draw [style=black edge] (49) to (38);
        		\draw [style=black edge] (50) to (38);
        		\draw [style=black edge] (56) to (45);
        		\draw [style=black edge] (57) to (46);
        		\draw [style=black edge] (58) to (47);
        		\draw [style=black edge] (48) to (53);
        		\draw [style=black edge] (49) to (54);
        		\draw [style=black edge] (50) to (55);
        		\draw [style=black edge] (59) to (64);
        		\draw [style=black edge] (64) to (61);
        		\draw [style=black edge] (61) to (62);
        		\draw [style=black edge] (62) to (63);
        		\draw [style=black edge] (70) to (59);
        		\draw [style=black edge] (71) to (59);
        		\draw [style=black edge] (72) to (59);
        		\draw [style=black edge] (73) to (63);
        		\draw [style=black edge] (74) to (63);
        		\draw [style=black edge] (75) to (63);
        		\draw [style=black edge] (81) to (70);
        		\draw [style=black edge] (82) to (71);
        		\draw [style=black edge] (83) to (72);
        		\draw [style=black edge] (73) to (78);
        		\draw [style=black edge] (74) to (79);
        		\draw [style=black edge] (75) to (80);
        		\draw [style=black edge] (84) to (89);
        		\draw [style=black edge] (89) to (86);
        		\draw [style=black edge] (86) to (87);
        		\draw [style=black edge] (87) to (88);
        		\draw [style=black edge] (95) to (84);
        		\draw [style=black edge] (96) to (84);
        		\draw [style=black edge] (97) to (84);
        		\draw [style=black edge] (98) to (88);
        		\draw [style=black edge] (99) to (88);
        		\draw [style=black edge] (100) to (88);
        		\draw [style=black edge] (106) to (95);
        		\draw [style=black edge] (107) to (96);
        		\draw [style=black edge] (108) to (97);
        		\draw [style=black edge] (98) to (103);
        		\draw [style=black edge] (99) to (104);
        		\draw [style=black edge] (100) to (105);
        		\draw [style=black edge] (109) to (112);
        		\draw [style=black edge] (112) to (113);
        		\draw [style=black edge] (110) to (114);
        		\draw [style=black edge] (114) to (115);
        		\draw [style=black edge] (111) to (116);
        		\draw [style=black edge] (116) to (117);
        		\draw [style=black edge] (0) to (118);
        		\draw [style=black edge] (118) to (111);
        		\draw [style=black edge] (34) to (119);
        		\draw [style=black edge] (119) to (111);
        		\draw [style=black edge] (34) to (120);
        		\draw [style=black edge] (120) to (109);
        		\draw [style=black edge] (59) to (121);
        		\draw [style=black edge] (121) to (109);
        		\draw [style=black edge] (84) to (122);
        		\draw [style=black edge] (122) to (109);
        		\draw [style=black edge] (84) to (123);
        		\draw [style=black edge] (123) to (110);
        		\draw [style=black edge] (4) to (124);
        		\draw [style=black edge] (124) to (110);
        		\draw [style=black edge] (38) to (125);
        		\draw [style=black edge] (125) to (110);
        		\draw [style=black edge] (63) to (126);
        		\draw [style=black edge] (126) to (111);
        		\draw [style=black edge] (153) to (154);
        		\draw [style=black edge] (154) to (0);
        		\draw [style=black edge] (155) to (156);
        		\draw [style=black edge] (156) to (34);
        		\draw [style=black edge] (157) to (158);
        		\draw [style=black edge] (158) to (34);
        		\draw [style=black edge] (159) to (160);
        		\draw [style=black edge] (160) to (59);
        		\draw [style=black edge] (161) to (162);
        		\draw [style=black edge] (162) to (84);
        		\draw [style=black edge] (163) to (164);
        		\draw [style=black edge] (164) to (84);
        		\draw [style=black edge] (4) to (165);
        		\draw [style=black edge] (165) to (166);
        		\draw [style=black edge] (38) to (167);
        		\draw [style=black edge] (167) to (168);
        		\draw [style=black edge] (63) to (169);
        		\draw [style=black edge] (169) to (170);
        	\end{pgfonlayer}
        \end{tikzpicture}
	 	\caption{The graph $G$ for $F = \left(x\lor y \lor z\right) \land \left(\bar x \lor z \lor \bar w\right) \land \left(x \lor \bar y \lor w \right)$}\label{fig4}
\end{figure}
Clearly, the graph $G$ can be computed from $F$ in polynomial time. 

First, suppose that $\mu$ is a truth assignment to the $n$ boolean variables
such that every clause of $F$ contains both a true and a false literal.
For every literal $\ell$, 
if $\ell$ is true under $\mu$, then let $c(v_\ell)=1$,
otherwise, let $c(v_\ell)=2$.
For every clause $C$, let $c(v_C)=1$.
For every vertex $u$ of degree $2$ in $G$, let $c(u)=1$.
For every literal $\ell$, 
let $c(\alpha_1^\ell)=1$ and 
let $c(\alpha_i^\ell)=2$ for every $i\in [d_\ell+3]\setminus \{ 1\}$.
For every clause $C$, 
if $C$ contains two true literals under $\mu$, then let $c(\alpha_C)=2$,
otherwise, let $c(\alpha_C)=1$.
It is easy to verify that $c$ is a majority additive $2$-coloring of $G$.

Next, suppose that $c$ is a majority additive $2$-coloring of $G$.
For every variable $x$, considering the vertex $b_x$ implies $c(v_x)\not=c(v_{\bar{x}})$.
For every clause $C=\ell_1\vee\ell_2\vee\ell_3$, 
considering the vertex $v_C$ implies that not all three values $c(\ell_i)$ are equal.
Therefore, setting a variable $x$ to true if and only if $c(v_x)=1$ yields a truth assignment 
such that every clause of $F$ contains both a true and a false literal.

This completes the proof for $k=2$.

\bigskip

\noindent For $k\geq 3$,
we reduce from the \textbf{NP}-complete problem \textsc{Graph $k$-Colorability}, cf. [GT4] in \cite{gajo}. 
Let $G$ be a graph of minimum degree at least $4$.
Let $G'$ be obtained from $G$ by replacing every edge $e=uv$ of $G$ 
with a path $ux_u^ey^ex_v^ev$ of order five,
that is, the graph $G'$ arises from $G$ by subdividing each edge three times.

First, suppose that $G$ has a vertex $k$-coloring $c_1:V(G)\to [k]$.
By Theorem 2 in \cite{bo},
the graph $G$ has a {\it majority $3$-edge-coloring} $c_2:E(G)\to [3]$,
that is, for every vertex $u$ of $G$, at most half the edges incident with $u$ have the same color under $c_2$.
Now, let 
$$c':V(G')\to [k]:
x\to \begin{cases}
c_1(x), &\text{ if } x\in V(G),\\
c_2(e), &\text{ if } x=y^e \text{ for some } e\in E(G)\mbox{, and }\\
1, &\text{ otherwise.}
\end{cases}$$
It is easy to verify that $c'$ is a majority additive $k$-coloring of $G'$.

Next, suppose that $c'$ is a majority additive $k$-coloring of $G'$.
Considering, for every edge $e=uv$, the vertex $y^e$ implies $c'(u)\not=c'(v)$.
Hence, restricting $c'$ to $V(G)$ yields a vertex $k$-coloring of $G$,
which completes the proof.
\end{proof}

\end{document}